\documentclass[a4paper,12pt]{amsart}
\usepackage{amsthm,amssymb}
\usepackage{stmaryrd}
\usepackage[T1]{fontenc}
\usepackage[utf8]{inputenc}
\usepackage{enumitem}
\usepackage{mathrsfs}
\usepackage{dsfont}
\usepackage{mathtools}
\usepackage{scalerel,stackengine}
\usepackage{nicefrac}
\usepackage{enumitem}
\usepackage[margin=3.5cm]{geometry}
\usepackage[numbers]{natbib}
\usepackage{caption}
\usepackage[all]{xy}

\numberwithin{equation}{section}

\usepackage[hypertexnames=false]{hyperref}
\hypersetup{
    colorlinks,
    linkcolor={red!90!black},
    citecolor={green!70!black},
    urlcolor={blue!80!black}
}

\usepackage{tikz}
\usetikzlibrary{calc,hobby}
\usetikzlibrary{positioning}
\usetikzlibrary{backgrounds}

\makeatletter
\DeclareSymbolFont{lettersA}{U}{txmia}{m}{it}
\SetSymbolFont{lettersA}{bold}{U}{txmia}{bx}{it}
\DeclareFontSubstitution{U}{txmia}{m}{it}

\DeclareMathSymbol{\m@thbbch@rA}{\mathord}{lettersA}{129}
\DeclareMathSymbol{\m@thbbch@rB}{\mathord}{lettersA}{130}
\DeclareMathSymbol{\m@thbbch@rC}{\mathord}{lettersA}{131}
\DeclareMathSymbol{\m@thbbch@rD}{\mathord}{lettersA}{132}
\DeclareMathSymbol{\m@thbbch@rE}{\mathord}{lettersA}{133}
\DeclareMathSymbol{\m@thbbch@rF}{\mathord}{lettersA}{134}
\DeclareMathSymbol{\m@thbbch@rG}{\mathord}{lettersA}{135}
\DeclareMathSymbol{\m@thbbch@rH}{\mathord}{lettersA}{136}
\DeclareMathSymbol{\m@thbbch@rI}{\mathord}{lettersA}{137}
\DeclareMathSymbol{\m@thbbch@rJ}{\mathord}{lettersA}{138}
\DeclareMathSymbol{\m@thbbch@rK}{\mathord}{lettersA}{139}
\DeclareMathSymbol{\m@thbbch@rL}{\mathord}{lettersA}{140}
\DeclareMathSymbol{\m@thbbch@rM}{\mathord}{lettersA}{141}
\DeclareMathSymbol{\m@thbbch@rN}{\mathord}{lettersA}{142}
\DeclareMathSymbol{\m@thbbch@rO}{\mathord}{lettersA}{143}
\DeclareMathSymbol{\m@thbbch@rP}{\mathord}{lettersA}{144}
\DeclareMathSymbol{\m@thbbch@rQ}{\mathord}{lettersA}{145}
\DeclareMathSymbol{\m@thbbch@rR}{\mathord}{lettersA}{146}
\DeclareMathSymbol{\m@thbbch@rS}{\mathord}{lettersA}{147}
\DeclareMathSymbol{\m@thbbch@rT}{\mathord}{lettersA}{148}
\DeclareMathSymbol{\m@thbbch@rU}{\mathord}{lettersA}{149}
\DeclareMathSymbol{\m@thbbch@rV}{\mathord}{lettersA}{150}
\DeclareMathSymbol{\m@thbbch@rW}{\mathord}{lettersA}{151}
\DeclareMathSymbol{\m@thbbch@rX}{\mathord}{lettersA}{152}
\DeclareMathSymbol{\m@thbbch@rY}{\mathord}{lettersA}{153}
\DeclareMathSymbol{\m@thbbch@rZ}{\mathord}{lettersA}{154}

\long\def\DoLongFutureLet #1#2#3#4{%
   \def\@FutureLetDecide{#1#2\@FutureLetToken
      \def\@FutureLetNext{#3}\else
      \def\@FutureLetNext{#4}\fi\@FutureLetNext}
   \futurelet\@FutureLetToken\@FutureLetDecide}
\def\DoFutureLet #1#2#3#4{\DoLongFutureLet{#1}{#2}{#3}{#4}}
\def\@EachCharacter{\DoFutureLet{\ifx}{\@EndEachCharacter}%
   {\@EachCharacterDone}{\@PickUpTheCharacter}}
\def\m@keCharacter#1{\csname\F@ntPrefix#1\endcsname}
\def\@PickUpTheCharacter#1{\m@keCharacter{#1}\@EachCharacter}
\def\@EachCharacterDone \@EndEachCharacter{}

\DeclareRobustCommand*{\varmathbb}[1]{\gdef\F@ntPrefix{m@thbbch@r}%
  \@EachCharacter #1\@EndEachCharacter}
\makeatother

\newtheorem{theorem}{Theorem}[section]
\newtheorem{lemma}[theorem]{Lemma}

\newtheoremstyle{mytheoremstyle} 
    {1em plus .2em minus .1em}                    
    {1em plus .2em minus .1em}                    
    {\rmfamily}                   
    {}                           
    {\bfseries}                   
    {.}                          
    {.5em}                       
    {}  

\theoremstyle{mytheoremstyle}

\newcommand\iffdef{\;\mathrel{\mathord{:}\mathord{\longleftrightarrow}}\;}

\DeclareMathOperator{\Int}{Int}

\DeclareMathOperator{\RCvtex}{RC}

\DeclareMathOperator{\ROvtex}{RO}
\newcommand{\con}{\mathrel{\mathsf{C}}}
\newcommand\mathbackslash{\raisebox{.4pt}{\texttt{/}}}
\def\notcon{
  \renewcommand\stacktype{L}\mathrel{\ensurestackMath{%
  \ThisStyle{\stackon[0pt]{\SavedStyle\con}{\SavedStyle\mathbackslash}}}}%
}

\title[On Nondefinability of Interior-Connectedness]{On Nondefinability of Interior-Connectedness\break via the Contact Relation}

\author[]{Rafa\l\ Gruszczy\'nski and Paula Mench\'{o}n}

\date{}

\address{Rafa\l\ Gruszczy\'nski, \textsc{Orcid:} 0000-0002-3379-0577\\
Paula Mench\'on, \textsc{Orcid:} 0000-0002-9395-107X\\
Department of Logic\\
Institute of Philosophy\\
Nicolaus Copernicus University in Toru\'n\\
Poland}

\email{gruszka@umk.pl, mpmenchon@nucompa.exa.unicen.edu.ar}

\begin{document}

\noindent \footnotesize{This is AAM of the paper published in Notre Dame J. Formal Logic 67(2): 291--303 (May 2026). DOI: 10.1215/00294527-2026-0003.}
\vspace{1cm}

\begin{abstract}
\begin{sloppypar}
This short paper is a small contribution to the field of Boolean Contact Algebras. We analyze the non-definability of the property of \emph{interior-connectedness}, and we prove certain minimality conditions for algebras and spaces that can be used in demonstrating that the aforementioned property cannot be expressed by means of contact within regular closed algebras.

\smallskip

\end{sloppypar}

\smallskip


\end{abstract}

\maketitle

\section{Introduction}
\label{sec1}

Boolean contact algebras (BCAs) are an algebraic framework for formalizing
spatial relations between objects that are often thought of as models of
natural phenomena, such as bodies or regions, understood as chunks of space.
It has been proved that they are essentially expansions of algebras of
regular closed (or open) subsets of topological spaces
(see \cite{Dimov-et-al-CARBTSPA1}; D\"{u}ntsch and Winter \cite{Duntsch-et-al-RTBCA}) and that their language
is strong enough to express various topological properties such as connectedness,
weak regularity, or clopenness (see Bennett and D\"{u}ntsch \cite{Bennett-Duntsch-AAT}).

Recall that a BCA is a pair
$\mathfrak{B}\coloneqq   \langle B,\mathrel{\mathsf{C}}
\rangle $, where $B$ is a Boolean algebra and $\mathrel{\mathsf{C}}$ is
a binary \emph{contact} relation on $B$ that satisfies the following five
axioms (where $+$ is the join operation, $\leq $ is the standard Boolean
order, and $\notcon $ is the complement of $\mathrel{\mathsf{C}}$):
%
\begin{gather}
(\forall x\in B)\,\mathbf{0}\notcon x, \tag{C0} \label{eqC0}
\\
(\forall x\in B)\,(x\neq \mathbf{0}\rightarrow x\mathrel{\mathsf{C}}x) ,
\tag{C1} \label{eqC1}
\\
(\forall x,y\in B)\,(x\mathrel{\mathsf{C}}y\rightarrow y\mathrel{ \mathsf{C}}x),
\tag{C2} \label{eqC2}
\\
(\forall x,y,z\in B)\,(x\mathrel{\mathsf{C}}y\wedge y\leq z \rightarrow x
\mathrel{\mathsf{C}}z), \tag{C3} \label{eqC3}
\\
(\forall x,y,z\in B)\,(x\mathrel{\mathsf{C}}y+x\rightarrow x\mathrel{
\mathsf{C}}y\vee x\mathrel{\mathsf{C}}z). \tag{C4} \label{C4} 
\end{gather}
These form the basic axioms for BCAs that are often extended with additional
constraints.

Let $  \langle X,\tau   \rangle $ be a topological space, where
$\tau $ is a family of open sets in $X$. For $A\subseteq X$,
$\overline{A}$ is its closure, and $\Int A$ is its interior. $A$ is a
\emph{regular closed} subset of $X$ if it is a fixed point of the composition
of closure and interior operations: $A=\overline{\Int A}$. Accordingly,
it is \emph{regular open} if it is a fixed point of the composition of interior
and closure: $A=\Int \overline{A}$. Both compositions are called
\emph{regularization} operations. $\RCvtex (\tau )$ (resp.
$\ROvtex (\tau )$) is the complete algebra of regular closed (resp. regular
open) subsets of $X$. The meet and complement of $\RCvtex (\tau )$ are defined
as, respectively, the regularization of the intersection and the closure
of the complement in the power set algebra $2^{X}$
\begin{equation*}
A\cdot D\coloneqq \overline{\Int (A\cap D)}\qquad {-}A\coloneqq \overline{X
\setminus A},
\end{equation*}
while the join is the standard set-theoretical union.

The canonical interpretation of BCAs treats $\mathfrak{B}$ as a subalgebra of
$\RCvtex (\tau )$, with the contact relation interpreted in the following way:
\begin{equation*}
A\mathrel{\mathsf{C}}_{\tau} D\iffdef A\cap D\neq \emptyset .
\end{equation*}
On the right, we have the set-theoretical intersection, which---in general---is
different from the meet operation of regular closed algebras. As is well known,
the contact operation lets us differentiate between two situations: when
two regions are disjoint and <<touch>> each other, and when two regions
are disjoint and <<separated>>, the distinction which is beyond the expressive
power of pure Boolean algebras. Thus, regular closed sets $A$ and
$B$ are in contact when they share at least one point, though they may
be disjoint in the sense of the meet operation of $\RCvtex (\tau )$.

Every nondegenerate (i.e., having at least two elements) Boolean algebra
$B$ carries two extreme contact relations, the \emph{minimal} that is just
the overlap relation
\begin{equation*}
x\mathrel{\mathsf{C}}y\iffdef x\cdot y\neq \mathbf{0},
\end{equation*}
and the \emph{maximal}, where any two nonzero regions are in contact
\begin{equation*}
x\mathrel{\mathsf{C}}y\iffdef x\neq \mathbf{0}\neq y.
\end{equation*}

Two contact algebras
$\mathfrak{B}_{1}\coloneqq   \langle B_{1},\mathrel{\mathsf{C}}_{1}
  \rangle $ and
$\mathfrak{B}_{2}\coloneqq   \langle B_{2},\mathrel{\mathsf{C}}_{2}
  \rangle $ are \emph{isomorphic} if there exists a~Boolean isomorphism
$f\colon B_{1}\to B_{2}$ that preserves and reflects the contact relation
\begin{equation*}
x\mathrel{\mathsf{C}}_{1} y\longleftrightarrow h(x)\mathrel{
\mathsf{C}}_{2} h(y).
\end{equation*}
Such an isomorphism is called a \emph{BCA-isomorphism}.

The representation theorems from \cite{Duntsch-et-al-RTBCA} and
\cite{Dimov-et-al-CARBTSPA1} show that all BCAs are (isomorphic to) subalgebras
of regular closed (equivalently: open) subsets of certain topological spaces.
Goldblatt and Grice \cite{Goldblatt-et-al-MADFMS} have shown that
the category of BCAs with contact reflecting Boolean homomorphisms is dually equivalent
to the category of mereotopological spaces and mereomorphisms.

Since the dawn of modern region-based investigations into the features
of space, their practitioners have delved into the problem of the expressive
power of languages with the binary predicate~$\mathrel{\mathsf{C}}$
(Pratt and Schoop \cite{Pratt-Schoop-EIPPM};
Bennett and D\"{u}ntsch \cite{Bennett-Duntsch-AAT};
Vakarelov \cite{Vakarelov-RBTS}). Parallel and
complementary investigations (Cohn et al. \cite{Cohn-et-al-QSRARQTRCC};
Pratt and Lemon \cite{Pratt-Lemon-OPPM}; Kontchakov et~al. \cite{Kontchakov-et-al-TCAML};
Pratt and Schoop \cite{Pratt-Schoop-EIPPM}; Pratt-Hartmann \cite{Pratt-Hartmann-EARIRBTOS};
Kontchakov et~al. \cite{Kontchakov-et-al-SLWCP}) tackled the problem of languages whose
signature contains unary predicates $c$ and $c^{\circ}$ (either the former
or both) interpreted topologically in the following way:
\begin{align*}
c(x)\qquad &\text{iff}\qquad \text{$x$ is a connected space,}
\\
{c^{\circ}}(x)\qquad &\text{iff}\qquad \text{the interior of $x$ is a connected space.}
\end{align*}
The second property bears the name of \emph{interior-connectedness} (or
\emph{internal connectedness}, as in the case of Ivanova's paper \cite{Ivanova-ECAATC}). Below,
we will use the former. Specifically, we say that a subset $A$ of a~topological
space $  \langle X,\tau   \rangle $ is
\emph{interior-connected} if its interior is a~connected subspace of
$X$.

Connectedness can be fairly easily defined by means of the contact in the
setting of BCAs via the following constraint
\begin{equation*}
c(x)\iffdef \bigl(\forall y,z\in B\setminus \{\mathbf{0}\}\bigr)\,(x=y+z
\rightarrow y\mathrel{\mathsf{C}}z).
\end{equation*}
In effect, in the regular closed algebras, for any topological space
$  \langle X,\tau   \rangle $ and any regular closed subset
$A$ of $X$ we obtain that
\begin{equation*}
\RCvtex (\tau )\vDash c(A)\qquad \text{iff}\qquad A\ \text{is connected.}
\end{equation*}
However, it is natural that the expressive power of the language of BCAs
must be topologically limited and that although many topological properties
can be captured with the contact relation, there will be such that cannot
be expressed in terms of it.

Interior-connectedness serves as an example, as was proved by Ivanova
\cite[Proposition~2.1]{Ivanova-ECAATC}. In her proof, Ivanova used Padoa's
method (see, e.g., Beth \cite{Beth-OPMITTOF}) and produced isomorphic contact
algebras $\mathfrak{B}_{1}$ and $\mathfrak{B}_{2}$ such that the first
has a regular closed set $A$ whose interior is connected but whose counterpart
in $\mathfrak{B}_{2}$ does not enjoy this property. The construction is
based on a topological space with seven points whose topology has certain
symmetries. It was our impression that the phenomenon of nondefinability
of the property of interior-connectedness can be analyzed in a~graph-based
setting by means of shapes and colors that enhance topological intuition
and understanding of the proof. In effect, we managed to simplify the proof
slightly and show that the proof must be minimal in the sense of the cardinalities
of both the spaces involved and their regular closed algebras.

Having said that, we want to emphasize that we do not use the graph topology
(see Hatcher \cite{Hatcher-AT}) in our investigations. The graphs with which
we represent spaces are heuristic devices, yet they are not deprived of
precise meaning that we of course explain.

In the sequel, we present our findings.

\section{The Proof}
\label{sec2}

\begin{figure}
\centering
\begin{tikzpicture}
    \coordinate (x) at (0, 0);
    \coordinate (y) at (4, 0);
    \coordinate (z) at ($(x)!0.5!(y) + (0, {2 * sqrt(3)})$);

    \coordinate (a) at ($(x)!0.5!(y)$); 
    \coordinate (yz) at ($(y)!0.5!(z)$); 

    \coordinate (Center) at ($(x)!2/3!(yz)$);

    \node[circle, fill=red,inner sep=2pt] at (x) {};
    \node[below=2pt] at (x) {$x$};
    \node[circle, fill=red,inner sep=2pt] at (y) {};
    \node[below=2pt] at (y) {$y$};
    \node[circle, fill=red,inner sep=2pt] at (z) {};
    \node[left=2pt] at (z) {$z$};

    \node[circle, fill,inner sep=2pt] at (a) {};

    \node[circle, fill,inner sep=2pt] at (Center) {};

    \node[below=2pt] at (a) {$a$};
    \node[above=2pt] at (Center) {$b$};


    \node[inner sep=5pt] (fakex) at (x) {};
    \node[inner sep=5pt] (fakey) at (y) {};
    \node[inner sep=5pt] (fakez) at (z) {};
    \node[inner sep=5pt] (fakea) at (a) {};

    \draw [red] (fakex) -- (fakea) -- (fakey) -- (fakez) -- (fakex);

    \path (fakex) edge [loop left,red] (fakex) (fakey) edge [loop right,red] (fakey) (fakez) edge [loop above,red] (fakez);

\end{tikzpicture}\caption{The five-point space.}\label{fig:5-point}
\end{figure}

Let us consider a~graph with five elements
$T\coloneqq \{x,y,z,a,b\}$, where $x$, $y$, and $z$ are distinguished nodes and $a$ and $b$ are non-distinguished nodes, and having seven
edges (including three loops), as in~Figure~\ref{fig:5-point}. The topology
$\tau $ we are going to consider is generated by the subbasis composed
of the three sides of the triangle: $\{x,z\}$, $\{y,z\}$,
$\{x,a,y\}$, and of the whole space~$T$.
Thus
\begin{equation*}
\tau =2^{\{x,y,z\}}\cup \bigl\{\{x,a,y\},\{x,a,y,z\}\bigr\}\cup \{T
\},
\end{equation*}
which, among others, means that there are eleven open sets in
$\tau $. Let us call \emph{distinguished} those sets all of whose elements belong to $\{x,y,z\}$. We ask the reader to observe that any open set in $T$ is composed
of points that are located on a path that begins with a distinguished node and ends
with a distinguished node (possibly the same). Clearly, $b$ is not an element of
any open proper subset of $T$, so the space is connected, and
$\emptyset $ and $T$ are its only clopen subsets. The two open proper subsets
of $T$ that are not distinguished are connected subspaces of $T$. $T$ is clearly
$T_{0}$ (but not $T_{1}$, due to $b$ again).

Let us determine the family $\ROvtex (\tau )$ of the regular open subsets of
$T$. Firstly, the operation of closure adds only non-distinguished nodes to distinguished singletons,
and so every such singleton is regular open. The closure of
$\{x,z\}$ consumes the two non-distinguished nodes, so its interior is again
$\{x,y\}$. Symmetrically, $\{z,y\}$ is regular open too. For
$\{x,a,y\}$, $b$ is its only limit point beyond the set, so again
$\{x,a,y\}$ is regular open. Thus, we have identified six elements of
$\ROvtex (\tau )$, and since the algebra cannot have more than eight elements (recall
that all in all there are eleven open sets) we have that
\begin{equation*}
\ROvtex (\tau )=\bigl\{\emptyset ,\{x\},\{y\},\{z\},\{x,z\},\{y,z\},\{x,a,y\},T
\bigr\}.
\end{equation*}

\begin{figure}
\centering
\begin{tikzpicture}
\newcommand{\cyan}[1]{\textcolor{cyan}{#1}}
 \node (0) at (0,0) {\cyan{$\emptyset$}};
    \node (x) at (-2,1.5) {\cyan{$\{x\}$}};
    \node (y) at (0,1.5) {\cyan{$\{y\}$}};
    \node (z) at (2,1.5) {\cyan{$\{z\}$}};
    \node (xy) at (-2,3) {$\{x,y\}$};
    \node (xz) at (0,3) {\cyan{$\{x,z\}$}};
    \node (yz) at (2,3) {\cyan{$\{y,z\}$}};
    \node (xay) at (-2,4.5) {\cyan{$\{x,a,y\}$}};
    \node (xayz) at (0,6) {$\{x,a,y,z\}$};
    \node (xyz) at (0,4.5) {$\{x,y,z\}$};
    \node (1) at (0,7.5) {\cyan{$\{x,y,z,a,b\}$}};
    \draw (0) -- (x) (0) -- (z) (0) -- (y) (x) -- (xy) (y) -- (xy) (x) -- (xz) (y) -- (yz) (z) -- (xz) (z) -- (yz)
    (xy) -- (xay) (xy) -- (xyz) (xay) -- (xayz) (xz) -- (xyz) (yz) -- (xyz) -- (xayz) -- (1);
    \begin{scope}[xshift=7.5cm]
  \node (0) at (0,0) {$\emptyset$};
  \node (xab) at (-2,1.5) {$\{x,a,b\}$};
  \node (yab) at (0,1.5) {$\{y,a,b\}$};
  \node (zb) at (2,1.5) {$\{z,b\}$};
  \node (xayb) at (-2,3) {$\{x,a,y,b\}$};
  \node (xzab) at (0,3) {$\{x,z,a,b\}$};
  \node (yzab) at (2,3) {$\{y,z,a,b\}$};
  \node (1) at (0,4.5) {$\{x,y,z,a,b\}$};
  \draw (0) -- (xab) (0) -- (yab) (0) -- (zb) (xab) -- (xayb) (xab) -- (xzab) (yab) -- (yzab)
            (yab) -- (xayb) (zb) -- (xzab) (zb) -- (yzab) (xayb) -- (1) (xzab) -- (1) (yzab) -- (1);
    \end{scope}
\end{tikzpicture}
\caption{On the left, we have the frame of open subsets of $T$ with its
regular elements being: the zero and the unity, all the singletons, $\{x,z\}$, $\{y,z\}$, and $\{x,a,y\}$; on the right, we have the Boolean algebra of regular
closed subsets of this space.}%
\label{fig:opens-of-T}
\end{figure}
\noindent In consequence
\begin{equation*}
\RCvtex (\tau )=\bigl\{\emptyset ,\{x,a,b\},\{y,a,b\},\{z,b\},\{x,z,a,b\},
\{y,z,a,b \},\{x,a,y,b\}, T\bigr\}.
\end{equation*}
Observe that any pair of nonempty sets in $\RCvtex (\tau )$ shares the point~$b$,
which means that the topological contact on the algebra is its maximal
contact relation, and the space $T$ is ultraconnected
(see Steen and Seebach~\cite[p.\,29]{Steen-Seebach-CT}). Figure~\ref{fig:opens-of-T} presents diagrammatically
both the frame of opens\footnote{Recall that in point-free
topology, the frame of opens of a topological space is its family of open
sets viewed as an algebraic structure. Basically, it is a complete Heyting
algebra (see Johnstone \cite{Johnstone-PPT}).} of $T$ and the Boolean algebra
$\RCvtex (\tau )$.

The regular closed set $R\coloneqq \{x,y,a,b\}$ (see Figure~\ref{fig:C-is-IC})
is interior-connected, as its interior $\{x,a,y\}$---being an open set
that has a non-distinguished point---is a connected subspace of~$T$.

\begin{figure}
\centering
\begin{tikzpicture}
\coordinate (a1) at (-.5,-.8);
\coordinate (a11) at (-.9,0.5);
\coordinate (a2) at (-.3,1);
\coordinate (a3) at (-.5,2.3);
\coordinate (a4) at (1,1.3);
\coordinate (a41) at (2.2,2.2);
\coordinate (a5) at (3,1.3);
\coordinate (a6) at (4.8,0);
\coordinate (a7) at (2,-0.8);
\fill[cyan,opacity=0.5,use Hobby shortcut] ([out angle=190]a1) .. (a11) .. ([in angle=180]a2) ..  (a4) .. (a41) .. ([in angle=0]a5) .. (a6) .. ([in angle=-10,out angle=180]a7) .. ([in angle=0]a1);
\node [left=5pt] at (a11) {$R$};
\draw [rounded corners,dashed] (-.5,-.6) rectangle (4.5,.5);
    \coordinate (x) at (0, 0);
    \coordinate (y) at (4, 0);
    \coordinate (z) at ($(x)!0.5!(y) + (0, {2 * sqrt(3)})$);
    \coordinate (yz) at ($(y)!0.5!(z)$);

    \coordinate (a) at ($(x)!0.5!(y)$); 

    \coordinate (Center) at ($(x)!2/3!(yz)$);

    \node[circle, fill=red,inner sep=2pt] at (x) {};
    \node[below=2pt] at (x) {$x$};
    \node[circle, fill=red,inner sep=2pt] at (y) {};
    \node[below=2pt] at (y) {$y$};
    \node[circle, fill=red,inner sep=2pt] at (z) {};
    \node[left=2pt] at (z) {$z$};

    \node[circle, fill,inner sep=2pt] at (a) {};

    \node[circle, fill,inner sep=2pt] at (Center) {};

    \node[below=2pt] at (a) {$a$};
    \node[above=2pt] at (Center) {$b$};


    \node[inner sep=5pt] (fakex) at (x) {};
    \node[inner sep=5pt] (fakey) at (y) {};
    \node[inner sep=5pt] (fakez) at (z) {};
    \node[inner sep=5pt] (fakea) at (a) {};

    \draw [red] (fakex) -- (fakea) -- (fakey) -- (fakez) -- (fakex);

 \path (fakex) edge [loop left,red] (fakex) (fakey) edge [loop right,red] (fakey) (fakez) edge [loop above,red] (fakez);


\end{tikzpicture}
\caption{$R$ is an interior-connected element of $\RCvtex (\tau )$.}%
\label{fig:C-is-IC}
\end{figure}

In the next step, we modify the space by deleting the non-distinguished point
$a$. Let $  \langle T\setminus \{a\},\tau '  \rangle $ be a subspace
of~$T$.

It is obvious that for any set $S\in \RCvtex (\tau )$,
$S\setminus \{a\}\in \RCvtex (\tau ')$. Moreover, this transformation does
not affect the point $b$, and so $\RCvtex (\tau ')$ has the maximal contact
relation. Thus the mapping $f\colon \RCvtex (\tau )\to \RCvtex (\tau ')$ such that
$f(S)\coloneqq S\setminus \{a\}$ is a BCA-isomorphism. Observe that
for $R$, $f(R)=R\setminus \{a\}$ is not interior-connected. Indeed,
$\Int _{\tau '}f(R)=\{x,y\}$, and it is a discrete subspace of
$T\setminus \{a\}$ (see Figures~\ref{fig:f(C)-is-not-IC} and \ref{fig:f-via-pic}).

Since we have the maximal contact relation in both cases, Ivanova proves
something a bit stronger, that is, the following theorem.
%
\begin{theorem}
\label{thm2.1}
The property of \emph{interior-connectedness} is not definable even in the
presence of the connectedness axiom for the contact relation
\begin{equation*}
\bigl(\forall x\in B\setminus \{\mathbf{0},\mathbf{1}\}\bigr)\,x\mathrel{
\mathsf{C}}-x.\footnote{Any BCA satisfies this axiom if and only if its corresponding
topological space is connected, hence its name (see Bennett and D\"{u}ntsch
 \cite{Bennett-Duntsch-AAT}).}
\end{equation*}
\end{theorem}

\begin{figure}
\begin{tikzpicture}
\coordinate (a1) at (-.5,-.8);
\coordinate (a11) at (-.9,0.5);
\coordinate (a2) at (-.3,1);
\coordinate (a3) at (-.5,2.3);
\coordinate (a4) at (1,1.3);
\coordinate (a41) at (2.2,2.2);
\coordinate (a5) at (3,1.3);
\coordinate (a6) at (4.8,0);
\coordinate (a7) at (2,-0.8);
\fill[cyan,opacity=0.5,use Hobby shortcut] ([out angle=190]a1) .. (a11) .. ([in angle=180]a2) ..  (a4) .. (a41) .. ([in angle=0]a5) .. (a6) .. ([in angle=-10,out angle=180]a7) .. ([in angle=0]a1);
\node [left=5pt] at (a11) {$R\setminus\{a\}$};
\draw [rounded corners,dashed] (-.5,-.6) rectangle (4.5,.5);
    \coordinate (x) at (0, 0);
    \coordinate (y) at (4, 0);
    \coordinate (z) at ($(x)!0.5!(y) + (0, {2 * sqrt(3)})$);

    \coordinate (a) at ($(x)!0.5!(y)$); 

    \coordinate (Center) at ($(x)!2/3!(yz)$);

    \node[circle, fill=red,inner sep=2pt] at (x) {};
    \node[below=2pt] at (x) {$x$};
    \node[circle, fill=red,inner sep=2pt] at (y) {};
    \node[below=2pt] at (y) {$y$};
    \node[circle, fill=red,inner sep=2pt] at (z) {};
    \node[left=2pt] at (z) {$z$};

    \node[circle, fill,inner sep=2pt] at (Center) {};

    \node[above=2pt] at (Center) {$b$};


    \node[inner sep=5pt] (fakex) at (x) {};
    \node[inner sep=5pt] (fakey) at (y) {};
    \node[inner sep=5pt] (fakez) at (z) {};
    \node[inner sep=5pt] (fakea) at (a) {};

    \draw [red] (fakex) -- (fakez) -- (fakey) -- (fakex);

 \path (fakex) edge [loop left,red] (fakex) (fakey) edge [loop right,red] (fakey) (fakez) edge [loop above,red] (fakez);

\end{tikzpicture}
\caption{$\Int _{\tau '}(R\setminus \{a\})$ is a discrete subspace of
$T\setminus \{a\}$, so $R\setminus \{a\}$ cannot be interior-connected.}%
\label{fig:f(C)-is-not-IC}
\end{figure}
Observe as well that since the algebras used in the proof are finite, restricting
the class to complete BCAs or even to complete BCAs satisfying the strong
version of \eqref{C4} from Gruszczy\'{n}ski and Mench\'{o}n \cite{Gruszczynski-Menchon-FCRTMOAB}
\begin{equation*}
x\mathrel{\mathsf{C}}
\bigvee _{i\in I}x_{i}\rightarrow (\exists i\in I)
\,x\mathrel{\mathsf{C}}x_{i}
\end{equation*}
has no influence on the status of definability of the interior-connectedness
property.

\begin{figure}
\centering
\begin{tikzpicture}
\newcommand{\redx}{\textcolor{red}{x}}
\newcommand{\redy}{\textcolor{red}{y}}
\newcommand{\redz}{\textcolor{red}{z}}
    \coordinate (ellipse) at (-2,3);
    \fill [cyan,opacity=0.5] (ellipse.center) ellipse (1cm and 0.4cm);
  \node (0) at (0,0) {$\emptyset$};
  \node (xab) at (-2,1.5) {$\{\redx,a,b\}$};
  \node (yab) at (0,1.5) {$\{\redy,a,b\}$};
  \node (zb) at (2,1.5) {$\{\redz,b\}$};
  \node (xayb) at (-2,3) {$\{\redx,a,\redy,b\}$};
  \node (xzab) at (0,3) {$\{\redx,\redz,a,b\}$};
  \node (yzab) at (2,3) {$\{\redy,\redz,a,b\}$};
  \node (1) at (0,4.5) {$\{\redx,\redy,\redz,a,b\}$};
  \draw (0) -- (xab) (0) -- (yab) (0) -- (zb) (xab) -- (xayb) (xab) -- (xzab) (yab) -- (yzab)
            (yab) -- (xayb) (zb) -- (xzab) (zb) -- (yzab) (xayb) -- (1) (xzab) -- (1) (yzab) -- (1);
    \coordinate (startf) at ($(yzab)!0.5!(1)$); 
    \begin{pgfonlayer}{background} 
     \fill [yellow,opacity=0.5] (xab.south east) ++(-15pt, 0pt) rectangle +(-22pt, 15pt);
    \fill [yellow,opacity=0.5] (yab.south east) ++(-15pt, 0pt) rectangle +(-22pt, 15pt);
    \end{pgfonlayer}{background}
\begin{scope}[xshift=7.5cm]
    \coordinate (ellipse) at (-2,3);
    \fill [cyan,opacity=0.5] (ellipse.center) ellipse (1cm and 0.4cm);
    \node (0) at (0,0) {$\emptyset$};
    \node (xb) at (-2,1.5) {$\{\redx,b\}$};
    \node (yb) at (0,1.5) {$\{\redy,b\}$};
    \node (zb) at (2,1.5) {$\{\redz,b\}$};
    \node (xyb) at (-2,3) {$\{\redx,\redy,b\}$};
    \node (xzb) at (0,3) {$\{\redx,\redz,b\}$};
    \node (yzb) at (2,3) {$\{\redy,\redz,b\}$};
    \node (1) at (0,4.5) {$\{\redx,\redy,\redz,b\}$};
    \draw (0) -- (xb) (0) -- (yb) (0) -- (zb) (xb) -- (xyb) (xb) -- (xzb) (yb) -- (yzb) (yb) -- (xyb) (zb) -- (xzb) (zb) -- (yzb) (xyb) -- (1) (xzb) -- (1) (yzb) -- (1);
    \coordinate (endf) at ($(xyb)!0.5!(1)$); 
    \begin{pgfonlayer}{background} 
    \fill [yellow,opacity=0.5] (xb.south east) ++(-15pt, 0pt) rectangle +(-13pt, 15pt);
    \fill [yellow,opacity=0.5] (yb.south east) ++(-15pt, 0pt) rectangle +(-13pt, 15pt);
    \end{pgfonlayer}{background}
    \end{scope}

    \draw [->,shorten <=15pt, shorten >=15pt,gray] (startf) to [out=45,in=135] node [above=3pt] {$f$} (endf);
\end{tikzpicture}
\caption{A pictorial presentation of the isomorphism
$f\colon \RCvtex (\tau )\to \RCvtex (\tau ')$. The highlighted rectangular chunks of regions form
interiors of the elements enclosed by an ellipse.}%
\label{fig:f-via-pic}
\end{figure}

\section{The Number of Elements of the Algebra}
\label{sec3}

Let $  \langle X_{1},\tau _{1}  \rangle $ and
$  \langle X_{2},\tau _{2}  \rangle $ be topological spaces, and
let $f\colon \RCvtex (\tau _{1})\to \RCvtex (\tau _{2})$ be a BCA-isomorphism.
We say that $f$ \emph{preserves interior-connectedness} just in case, for
all $b \in \RCvtex (\tau _{1})$, $b$ has a connected interior in $X_{1}$ if and only if
$f(b)$ has a connected interior in $X_{2}$.

We will show that the proof of Ivanova is the simplest possible in the
following sense.
%
\begin{sloppypar}
\begin{theorem}%
\label{th:8-elements}
Let $  \langle X_{1},\tau _{1}  \rangle $ and
$  \langle X_{2},\tau _{2}  \rangle $ be topological spaces. If
$f$ is a BCA-isomorphism of the contact algebras
$\mathfrak{B}_1\coloneqq  \langle \RCvtex (\tau _{1}),\mathord{\mathrel{\mathsf{C}}_{\tau _{1}}}
  \rangle $ and
$\mathfrak{B}_2\coloneqq  \langle \RCvtex (\tau _{2}),\mathord{\mathrel{\mathsf{C}}_{\tau _{2}}}
  \rangle $, but $f$ does not preserve interior-connectedness, then
both algebras must have at least eight elements.
\end{theorem}
\end{sloppypar}
\begin{proof}
The algebras cannot have two elements each since, in this case, their unities
are the whole topological spaces $X_{1}$ and $X_{2}$. For the whole space,
interior-connectedness coincides with mere connectedness. So, it would
have to be the case that $X_{1}$ is connected and $X_{2}$ is not (or vice
versa). But in such a situation, $X_{2}$ must be decomposable into two
clopen nonempty sets, and so $\mathfrak{B}_{2}$ must have at least four
elements (as every clopen set is regular closed).

Suppose that there exist two isomorphic contact algebras
$\langle \RCvtex (\tau _{1}),\mathord{\mathrel{\mathsf{C}}_{\tau _{1}}}
  \rangle $ and
$\langle \RCvtex (\tau _{2}),\mathord{\mathrel{\mathsf{C}}_{\tau _{2}}}
  \rangle $ with only four elements that
are isomorphic and for which there is an BCA-isomorphism
$f\colon\RCvtex (\tau _{1}) \to \RCvtex (\tau _{2})$ such that there exists a set $A \in \RCvtex (\tau _{1})$ which
is interior-connected in $X_{1}$, but
$f(A)$ (the image of $A$ with respect to $f$) is not interior-connected
in $X_{2}$, that is,
$\Int _{\tau _{2}}f(A)$ is a disconnected subspace of $X_{2}$.
Thus, there exist two nonempty open sets $D'$, $E'$ of the subspace
$\Int _{\tau _{2}} f(A)$ such that
\begin{equation*}
\Int _{\tau _{2}} f(A) = D' \cup E' \qquad \text{and} \qquad D' \cap E'
= \emptyset .
\end{equation*}
Since $\Int _{\tau _{2}} f(A)$ is open in $X_{2}$, $D'$ and $E'$ are also
open in $X_{2}$. It follows that
\begin{equation*}
f(A) = \overline{\Int _{\tau _{2}}f(A)} = \overline{D' \cup E'} =
\overline{D'} \cup \overline{E'},
\end{equation*}
where the closure operation is the operation of the space $X_{2}$. Being
the closures of nonempty open sets, $\overline{D'}$ and
$\overline{E'}$ are nonempty regular closed sets. Since
$D' \cap E' = \emptyset $ and $D'$ and $E'$ are open sets, it follows that
$\overline{D'} \neq \overline{E'}$. As $\mathfrak{B}_{2}$ has only four
elements, it must be the case that $f(A)=X_{2}$, and so ($\dagger $)
$A=X_{1}$.

Observe that $\overline{D'} \cap \overline{E'} = \emptyset $. Suppose toward
a contradiction that there exists
$x \in \overline{D'} \cap \overline{E'}$. Since $X_{2}$ is the union
$D' \cup E'$, $x$ is in one of these sets, and without loss of generality,
we can assume that $x \in D'$. Since $D'$ is an open set and
$x \in \overline{E'}$, it follows that $D' \cap E' \neq \emptyset $, a~contradiction.
Therefore, $\overline{D'}$ is not in contact with $\overline{E'}$.

Let $D, E \in \RCvtex (\tau _{1})$ be such that $f(D) = \overline{D'}$ and
$f(E) = \overline{E'}$. By assumption, $D$ is not in contact with
$E$, that is, $D \cap E = \emptyset $. The algebra $\RCvtex (\tau _{1})$ has
only four elements, so $X_{1}=D\cup E$, which means that $X_{1}$ is disconnected,
and so its interior is disconnected, too (since this is the whole space).
But this contradicts $(\dagger )$.
\end{proof}

\section{The Number of Points}
\label{sec4}

In this section, we will prove the following minimality theorem.
%
\begin{theorem}%
\label{th:minimality-theorem}
Let $  \langle X_{1},\tau _{1}  \rangle $ and
$  \langle X_{2},\tau _{2}  \rangle $ be topological spaces. If
$f$ is a BCA-isomorphism of the algebras
$\mathfrak{B}_{1}\coloneqq   \langle \RCvtex (\tau _{1}),\mathord{
\mathrel{\mathsf{C}}_{\tau _{1}}}  \rangle $ and
$\mathfrak{B}_{2}\coloneqq   \langle \RCvtex (\tau _{2}),\mathord{
\mathrel{\mathsf{C}}_{\tau _{2}}}  \rangle $, but $f$ does not preserve
interior-connectedness, then the minimal number of points of the underlying
spaces is five for one of the spaces and four for the other.
\end{theorem}
 In this way, we will show that Ivanova's construction reduced to a five-point
space is the minimal possible construction, not only in the sense of the
cardinality of algebras but also in the sense of the cardinality of sets
of points of the underlying spaces. The following lemma that stems from
Theorem~\ref{th:8-elements} will be crucial.
%
\begin{lemma}%
\label{lem:4-points}
If $\mathfrak{B}_{1}$ and $\mathfrak{B}_{2}$ and $f$ are like in Theorem~\ref{th:minimality-theorem},
then each of the underlying topological spaces must have at least four
points.
\end{lemma}
\begin{proof}
Assume we have isomorphic BCAs $\mathfrak{B}_{1}$ and
$\mathfrak{B}_{2}$ that have different properties of interior-connectedness.
Let $f\colon \mathfrak{B}_{1}\to \mathfrak{B}_{2}$ be their BCA-isomorphism.
Suppose the first is the algebra of regular closed sets of a space
$\langle X_{1},\tau _{1}\rangle $ whose domain has precisely three points:
$\mathfrak{B}_{1}\coloneqq   \langle \RCvtex (\tau _{1}),\mathrel{
\mathsf{C}}_{1}  \rangle $. In light of Theorem~\ref{th:8-elements},
it must be the case that $\RCvtex (\tau _{1})=2^{X_{1}}$, so all subsets of
$X_{1}$ are closed, and $X_{1}$ is a discrete space. This also means that
$\RCvtex (\tau _{1})$ has the minimal contact relation. In this case, the only
nonempty interior-connected subsets of $X_{1}$ are the singletons.

Suppose $\langle X_{2},\tau _{2}\rangle $ is a topological space whose
regular closed algebra
$\mathfrak{B}_{2}\coloneqq   \langle \RCvtex (\tau _{2}),\mathord{
\mathrel{\mathsf{C}}_{2}}  \rangle $ is isomorphic to
$\mathfrak{B}_1$, that is, $\RCvtex (\tau _{2})$ has exactly eight elements.
Since $f$ is an isomorphism of contact algebras, we have that
\begin{equation*}
A\cap D\neq \emptyset \qquad \text{iff}\qquad f(A)\mathrel{
\mathsf{C}}_{2} f(D).
\end{equation*}
Let $\cdot $ be the operation of meet of $\RCvtex (\tau _{2})$. Then we have
that $f(A)\mathrel{\mathsf{C}}_{2} f(D)$ entails that
$A\cap D\neq \emptyset $, so $f(A)\cdot f(D)\neq \emptyset $, as $f$ is
a Boolean homomorphism. This implies that the contact relation of
$\RCvtex (\tau _{2})$ must also be minimal. But this means that
$\RCvtex (\tau _{2})$ coincides with the algebra of clopen subsets of the space
$X_{2}$. So, its atoms are the only connected components of the space that
are also interior-connected. Therefore, $\mathfrak{B}_{1}$ and
$\mathfrak{B}_{2}$ have the same notions of interior-connectedness, which
is a contradiction. So, both spaces must have at least four points each.
This ends the proof.
\end{proof}

\begin{sloppypar}
We can see that the constructions so far are almost minimal, as we have
a BCA-isomorphism between two eight-element regular closed algebras, where
the first one is based on the space with five points, and the second on
the space of four points. We will show that we cannot reduce the cardinality
of the first space to four points, and in this way we will prove Theorem~\ref{th:minimality-theorem}.
\end{sloppypar}

Firstly, let us observe that in the case we begin with the regular closed
algebra of a~four-point space $\langle X,\tau \rangle $ we cannot consider
the regular contact algebra of sixteen elements. This is because in such a
situation we would have that $\RCvtex (\tau )=2^{X}$, the space is discrete,
and the only nonempty interior-connected sets are the singletons. So any
isomorphic algebra of another four-point space would also have to be the
whole power set algebra, and would have the same topological properties
as $\RCvtex (\tau )$. Thus, we may focus on eight-element algebras of four-point
spaces.

If we want to construct a contact relation on the algebra of eight elements,
we have three possibilities for the atoms:
\begin{enumerate}
\item[1.] all three atoms are pairwise in contact,
\item[2.] one pair of atoms is in contact,
\item[3.] contact is empty on atoms.
\end{enumerate}
We will consider the four-element set $S\coloneqq \{p,q,r,s\}$ and we analyze the three above-mentioned cases on $S$ as the domain of spaces. Since the algebras we have in mind are regular closed algebras, the contact must be defined by set-theoretical intersection.

Observe that it indeed cannot be the case that exactly two pairs of atoms are in contact, since then there would have to be a pair of atoms that add up to the whole space, as in~Figure~\ref{fig:impossible-atoms}, which is impossible.

In the first case, the only possibility
for three atoms to be pairwise in topological contact is that they have
a point in common; let it be $s$, as in Figure~\ref{fig:with-Heyting}.
In the same figure, the diagram on the right represents the frame structure
of the topology $\tau $ determined by the algebra. From it, we can see
that $  \langle S,\tau   \rangle $ is---up to homeomorphism---precisely
the space $T\setminus \{a\}$ in Figure~\ref{fig:f(C)-is-not-IC}, whose
only nonempty interior-connected subsets are atoms and the unity. By Lemma~\ref{lem:4-points},
the isomorphic image of $\RCvtex (\tau )$---if we limit ourselves to only at
most four-point spaces---must have the maximal contact, and so the contact
must hold among all pairs of atoms. Then, it will again be homeomorphic
to $T\setminus \{a\}$, and so will have no different interior-connected
regular closed sets, and thus there is no way to map any of the interior-connected
sets to a set that does not have the property. In consequence, in the case
of all three atoms in contact, at least one of the spaces involved must
have at least five points.

\begin{figure}
    \centering
\begin{tikzpicture}
    \node (0) at (0,0) {$\emptyset$};
  \node (pq) at (-2,1.5) {$\{p,q\}$};
  \node (pr) at (0,1.5) {$\{p,r\}$};
  \node (rs) at (2,1.5) {$\{r,s\}$};
  \node (1) at (0,4.5) {$\{p,q,r,s\}$};
  \coordinate (apq) at (-2,2.5);
  \coordinate (apr) at (0,2.5);
  \coordinate (ars) at (2,2.5);
  \draw (0) -- (pq) (0) -- (pr) (0) -- (rs);
  \draw [dotted,very thick,shorten >=5pt] (pq) -- (apq);
  \draw [dotted,very thick,shorten >=5pt] (pr) -- (apq);
  \draw [dotted,very thick,shorten >=5pt] (pr) -- (ars);
  \draw [dotted,very thick,shorten >=5pt] (rs) -- (ars);
  \draw (pq) -- (1) (rs) -- (1);
\end{tikzpicture}
\caption{An impossible configuration of atoms on the four-point space
$S$: atoms $\{p,q\}$ and $\{r,s\}$ add up to $S$.}%
\label{fig:impossible-atoms}
\end{figure}

\begin{figure}
\centering
\begin{tikzpicture}
 \node (0) at (0,0) {$\emptyset$};
    \node (ps) at (-2,1.5) {$\{p,s\}$};
    \node (qs) at (0,1.5) {$\{q,s\}$};
    \node (rs) at (2,1.5) {$\{r,s\}$};
    \node (pqs) at (-2,3) {$\{p,q,s\}$};
    \node (prs) at (0,3) {$\{p,r,s\}$};
    \node (qrs) at (2,3) {$\{q,r,s\}$};
    \node (1) at (0,4.5) {$\{p,q,r,s\}$};
    \draw (0) -- (ps) (0) -- (qs) (0) -- (rs) (ps) -- (pqs) (ps) -- (prs) (qs) -- (qrs) (qs) -- (pqs) (rs) -- (prs) (rs) -- (qrs) (pqs) -- (1) (prs) -- (1) (qrs) -- (1);
\begin{scope}[xshift=7.5cm]
\node (0) at (0,0) {$\emptyset$};
\node (p) at (-2,1.5) {$\{p\}$};
\node (q) at (0,1.5) {$\{q\}$};
\node (r) at (2,1.5) {$\{r\}$};
\node (pq) at (-2,3) {$\{p,q\}$};
\node (pr) at (0,3) {$\{p,r\}$};
\node (qr) at (2,3) {$\{q,r\}$};
\node (pqr) at (0,4.5) {$\{p,q,r\}$};
\node (1) at (0,6) {$\{p,q,r,s\}$};
\draw (0) -- (p) (0) -- (q) (0) -- (r) (p) -- (pq) (p) -- (pr) (q) -- (pq) (q) -- (qr) (r) -- (pr) (r) -- (qr) (pq) -- (pqr) (pr) -- (pqr) -- (1) (qr) -- (pqr);
\end{scope}
\end{tikzpicture}
\caption{The eight-element regular closed algebra with full contact and
determined by it the frame of open subsets of $S$. We can see that
$S$ is---up to homeomorphism---the space $T\setminus \{a\}$ in Figure~\ref{fig:f(C)-is-not-IC}.}%
\label{fig:with-Heyting}
\end{figure}

\begin{figure}
\centering
\begin{tikzpicture}
 \node (0) at (0,0) {$\emptyset$};
    \node (p) at (-2,1.5) {$\{p\}$};
    \node (qs) at (0,1.5) {$\{q,s\}$};
    \node (rs) at (2,1.5) {$\{r,s\}$};
    \node (pqs) at (-2,3) {$\{p,q,s\}$};
    \node (prs) at (0,3) {$\{p,r,s\}$};
    \node (qrs) at (2,3) {$\{q,r,s\}$};
    \node (1) at (0,4.5) {$\{p,q,r,s\}$};
    \draw (0) -- (p) (0) -- (qs) (0) -- (rs) (p) -- (pqs) (p) -- (prs) (qs) -- (pqs) (qs) -- (qrs) (rs) -- (prs) (rs) -- (qrs) (pqs) -- (1) (prs) -- (1) (qrs) -- (1);
\begin{scope}[xshift=7.5cm]
\node (0) at (0,0) {$\emptyset$};
\node (r) at (-2,1.5) {$\{r\}$};
\node (q) at (0,1.5) {$\{q\}$};
\node (p) at (2,1.5) {$\{p\}$};
\node (qr) at (-2,3) {$\{q,r\}$};
\node (qrs) at (-2,4.5) {$\{q,r,s\}$};
\node (pr) at (0,3) {$\{p,r\}$};
\node (pq) at (2,3) {$\{p,q\}$};
\node (pqr) at (0,4.5) {$\{p,q,r\}$};
\node (1) at (0,6) {$\{p,q,r,s\}$};
\draw (0) -- (r) (0) -- (q) (0) -- (p) (p) -- (pq) (p) -- (pr) (q) -- (pq) (q) -- (qr) (r) -- (pr) (r) -- (qr) (qr) -- (qrs) (qr) -- (pqr) (pq) -- (pqr) (pr) -- (pqr) -- (1) (qrs) -- (1);
\end{scope}
\end{tikzpicture}
\caption{The regular closed algebra of $S$ with one pair of atoms in contact,
$\{q,s\}$ and $\{r,s\}$; and the frame of opens determined by the algebra.
We can see that $\{q,r,s\}$ is an interior-connected clopen set.}%
\label{fig:2-atoms-intersect}
\end{figure}

\begin{figure}
\centering
\begin{tikzpicture}
 \node (0) at (0,0) {$\emptyset$};
    \node (p) at (-2,1.5) {$\{p\}$};
    \node (q) at (0,1.5) {$\{q\}$};
    \node (rs) at (2,1.5) {$\{r,s\}$};
    \node (pq) at (-2,3) {$\{p,q\}$};
    \node (prs) at (0,3) {$\{p,r,s\}$};
    \node (qrs) at (2,3) {$\{q,r,s\}$};
    \node (1) at (0,4.5) {$\{p,q,r,s\}$};
    \draw (0) -- (p) (0) -- (q) (0) -- (rs) (p) -- (pq) (p) -- (prs) (qs) -- (pq) (qs) -- (qrs) (rs) -- (prs) (rs) -- (qrs) (pq) -- (1) (prs) -- (1) (qrs) -- (1);
\end{tikzpicture}
\caption{The regular closed algebra on the four-point set
$\{p,q,r,s\}$ with all atoms separated. Routine verification shows that
$\RCvtex (\tau )=\ROvtex (\tau )=\tau $, so this is the algebra of clopen subsets
of the space.}%
\label{fig:clopens}
\end{figure}

In the second case, where only one pair of atoms is in contact, that is, the situation
in Figure~\ref{fig:2-atoms-intersect} holds. Here, we have only one interior-connected
set that is neither the empty set, nor an atom, nor the whole space, that is,
$\{q,r,s\}$ (interestingly, this set is clopen).
As---up to isomorphism---this is the only regular contact algebra in which
only one pair of atoms is in contact, the only possibility to <<spoil>>
interior-connectedness of $\{q,r,s\}$ is via an automorphism. Since such
an automorphism must, in particular, preserve contact on atoms, the only
two possibilities are the identity relation or the function $f$ that swaps
$\{q,s\}$ for $\{r,s\}$ (and vice versa), and sends $\{p\}$ to
$\{p\}$. However, the second possibility cannot hold, since if we put
$f(\{q,s\})\coloneqq \{r,s\}$ and $f(\{r,s\})\coloneqq \{q,s\}$, in consequence
we must have $f(\{p,q,s\})=\{q,r,s\}$ and $f(\{q,r,s\})=\{p,q,s\}$. But
then such an $f$ fails to be an isomorphism, as
$\{p\}\subseteq \{p,r,s\}$ but $f(\{p\})\nsubseteq f(\{p,r,s\})$. So, the
identity mapping is the only option. Therefore, also in this case at least
one of the two algebras must have more than four points.

The last situation is when the contact relation is empty on the set of
atoms. Since the atoms must add up to the unity, in the case of four-point
space and eight-element algebra, the only possibility is that we have two
singletons, say $\{p\}$ and $\{q\}$, and one two-element set
$\{r,s\}$. In this case, we obtain the regular closed algebra in Figure~\ref{fig:clopens},
whose topological space has the same algebra of regular open sets, and
the same frame of opens, which is the algebra of clopen subsets of the
space. Up to isomorphism, it is the only regular closed eight-element algebra
so automorphisms are the only possibility. As no nontrivial set is interior-connected,
we will not show nondefinability of interior-connectedness using this
algebra.

In this way, we proved Theorem~\ref{th:minimality-theorem}.

\section*{Acknowledgments}

We would like to thank Ivo D\"{u}ntsch for reading various versions of
this paper and for stimulating discussions concerning the problems addressed
above. With his help, we were able to significantly improve the presentation.
We are also grateful to the anonymous referee whose comments and suggestions
contributed to improving this paper. Any remaining errors or shortcomings
are entirely our responsibility.

\begin{sloppypar}
This research was funded by the National Science Center (Poland), grant
number 2020/39/B/HS1/00216.
\end{sloppypar}





\begin{thebibliography}{10}

\bibitem[Bennett and D\"{u}ntsch, 2007]{Bennett-Duntsch-AAT}
Bennett, B., and I. D\"{u}ntsch,
\newblock ``Axioms, algebras and topology''
\newblock pp. 99--159 in \emph{Handbook of Spatial Logics}, edited by M. Aiello, I. Pratt-Hartmann, and J. Van~Benthem, Springer, Dordrecht, 2007.
\newblock DOI: \url{10.1007/978-1-4020-5587-4\_3}, MR2393887.


\bibitem[Beth, 1953]{Beth-OPMITTOF}
Beth, E.,
\newblock ``On Padoa's method in the theory of definition,''
\newblock \emph{Indagationes Mathematicae (Proceedings)}, vol. 56 (1953), pp. 330--39.
\newblock MR0058537.

\bibitem[Cohn et~al., 1997]{Cohn-et-al-QSRARQTRCC}
Cohn, A.~G., B. Bennett, J. Gooday, and N.~M. Gotts,
\newblock ``Qualitative spatial representation and reasoning with the region connection calculus,''
\newblock \emph{Geoinformatica}, vol. 1 (1997), no. 3, pp. 275--316.


\bibitem[Dimov and Vakarelov, 2006]{Dimov-et-al-CARBTSPA1}
Dimov, G. D., and D. Vakarelov,
\newblock ``Contact algebras and region-based theory of space: A proximity approach--I,''
\newblock \emph{Fundamenta Informaticae}, vol. 74 (2006), nos. 2--3, pp. 209--49.
\newblock DOI: \url{10.3233/fun-2006-742-303}, MR2284194.

\bibitem[D\"{u}ntsch and Winter, 2005]{Duntsch-et-al-RTBCA}
D\"{u}ntsch, I., and M. Winter,
\newblock ``A representation theorem for Boolean contact algebras,''
\newblock \emph{Theoretical Computer Science}, vol. 347 (2005), no. 3, pp. 498--512.
\newblock DOI: \url{10.1016/j.tcs.2005.06.030}, MR2187916.


\bibitem[Goldblatt and Grice, 2016]{Goldblatt-et-al-MADFMS}
Goldblatt, R., and M. Grice,
\newblock ``Mereocompactness and duality for mereotopological spaces,''
\newblock pp. 313--30 in \emph{J. Michael Dunn on Information Based Logics},
edited by K. Bimb\'o, Springer, Cham, 2016.
\newblock DOI: \url{10.1007/978-3-319-29300-4_15}, MR3526510.


\bibitem[Gruszczy\'{n}ski and Mench\'{o}n, 2023]{Gruszczynski-Menchon-FCRTMOAB}
Gruszczy\'{n}ski, R., and P. Mench\'{o}n,
\newblock ``From contact relations to modal operators, and back,''
\newblock \emph{Studia Logica}, vol. 111 (2023), no. 5, pp. 717--48.
\newblock DOI: \url{10.1007/s11225-023-10036-7}, MR4646444.

\bibitem[Hatcher, 2002]{Hatcher-AT}
Hatcher, A.,
\newblock \emph{Algebraic Topology},
\newblock Cambridge University Press, Cambridge, 2002.
MR1867354

\bibitem[Ivanova, 2020]{Ivanova-ECAATC}
Ivanova, T.,
\newblock ``Extended contact algebras and internal connectedness,''
\newblock \emph{Studia Logica}, vol. 108 (2020), no. 2, pp. 239--54.
\newblock DOI: \url{10.1007/s11225-019-09845-6}, MR4079271.


\bibitem[Johnstone, 1983]{Johnstone-PPT}
Johnstone, P.~T.,
\newblock ``The point of pointless topology,''
\newblock \emph{Bulletin (New Series) of the American Mathematical Society}, vol. 8 (1983), no. 1, pp. 41--53.
\newblock DOI: \url{10.1090/S0273-0979-1983-15080-2}, MR0682820.

\bibitem[Kontchakov et~al., 1998]{Kontchakov-et-al-TCAML}
Kontchakov, R., I. Pratt-Hartmann, F. Wolter, and M. Zakharyaschev,
\newblock ``Topology, connectedness, and modal logic''
\newblock pp. 151--76 in \emph{Advances in Modal Logic}, edited by M. Kracht, M. de~Rijke, H. Wansing, and M. Zakharyaschev, CSLI Publications, 1998.
\newblock MR2642643

\bibitem[Kontchakov et~al., 2010]{Kontchakov-et-al-SLWCP}
Kontchakov, R., I. Pratt-Hartmann, F. Wolter, and M. Zakharyaschev,
\newblock ``Spatial logics with connectedness predicates,''
\newblock \emph{Logical Methods in Computer Science}, vol. 6 (2010), no. 3.
\newblock DOI: \url{10.2168/LMCS-6(3:7)2010}, MR2679071.

\bibitem[Pratt and Lemon, 1997]{Pratt-Lemon-OPPM}
Pratt, I., and O. Lemon,
\newblock ``Ontologies for plane, polygonal mereotopology,''
\newblock \emph{Notre Dame Journal of Formal Logic}, vol. 38 (1997), no. 2, pp. 225--45.
\newblock DOI: \url{10.1305/ndjfl/1039724888}, MR1489411


\bibitem[Pratt and Schoop, 2000]{Pratt-Schoop-EIPPM}
Pratt, I., and D. Schoop,
\newblock ``Expressivity in polygonal, plane mereotopology,''
\newblock \emph{Journal of Symbolic Logic}, vol. 65 (2000), no. 2. pp. 822--38.
\newblock DOI: \url{10.2307/2586573}, MR1771089.


\bibitem[Pratt-Hartmann, 2001]{Pratt-Hartmann-EARIRBTOS}
Pratt-Hartmann, I.,
\newblock ``Empiricism and rationalism in region-based theories of space,''
\newblock \emph{Fundamenta Informaticae}, vol. 46 (2001), nos. 1--2, pp. 159--86.
\newblock MR2009806.

\bibitem[Steen and Seebach, 1995]{Steen-Seebach-CT}
Steen, L.~A., and J. A. Seebach,
\newblock \emph{Counterexamples in Topology}.
\newblock Dover Publications, New York, 1995.
\newblock MR1382863


\bibitem[Vakarelov, 2007]{Vakarelov-RBTS}
Vakarelov, D.,
\newblock ``Region-based theory of space: Algebras of regions, representation theory, and logics''
\newblock pp. 267--348 in \emph{Mathematical Problems from Applied Logic II: Logics for the XXIst Century}, edited by D.~M. Gabbay, M. Zakharyaschev, and S.~S. Goncharov, Springer, New York, 2007.
\end{thebibliography}
\end{document}